\documentclass[12pt]{article}

\usepackage{enumerate}
\usepackage{amsmath}
\usepackage{amssymb,latexsym}
\usepackage{amsthm}
\usepackage{color}

\usepackage{graphicx}

\usepackage{hyperref}

\usepackage{amscd}

\DeclareMathOperator{\C}{\mathcal{C}}

\newtheorem{theorem}{Theorem}[section]

\newtheorem{lemma}[theorem]{Lemma}
\newtheorem{corollary}[theorem]{Corollary}
\newtheorem{definition}[theorem]{Definition}
\newtheorem{proposition}[theorem]{Proposition}

\newtheorem{remark}[theorem]{Remark}

\newcommand{\cC}{{\mathcal C}}

\newcommand{\cL}{{\mathcal L}}

\newcommand{\F}{{\mathbb F}}
\newcommand{\fq}{{\mathbb F}_q}
\newcommand{\fqn}{{\mathbb F}_{q^n}}

\newcommand{\la}{\langle}
\newcommand{\ra}{\rangle}

\newcommand{\PG}{\mathrm{PG}}
\newcommand{\N}{\mathrm{N}}

\title{On certain linearized polynomials with high degree and kernel of small dimension}
\author{Olga Polverino, Giovanni Zini and Ferdinando Zullo\thanks{
This research was supported by the project ``VALERE: VAnviteLli pEr la RicErca" of the University of Campania ``Luigi Vanvitelli'', and by the Italian National Group for Algebraic and Geometric Structures and their Applications (GNSAGA
- INdAM). }}

\begin{document}
\maketitle

\begin{abstract}
Let $f$ be the $\fq$-linear map over $\F_{q^{2n}}$ defined by $x\mapsto x+ax^{q^s}+bx^{q^{n+s}}$ with $\gcd(n,s)=1$. It is known that the kernel of $f$ has dimension at most $2$, as proved by Csajb\'ok et al. in ``A new family of MRD-codes'' (2018).
For $n$ big enough, e.g. $n\geq5$ when $s=1$, we classify the values of $b/a$ such that the kernel of $f$ has dimension at most $1$.
To this aim, we translate the problem into the study of some algebraic curves of small degree with respect to the degree of $f$; this allows to use intersection theory and function field theory together with the Hasse-Weil bound.
Our result implies a non-scatteredness result for certain high degree scattered binomials, and the asymptotic classification of a family of rank metric codes.
\end{abstract}



\section{Introduction}
Let $q$ be a prime power and let $m$ be a positive integer.
A $q$-\emph{polynomial}, or \emph{linearized polynomial}, over $\F_{q^m}$ is a polynomial of the form
\[f(x)=\sum_{i=0}^t a_i x^{q^i},\]
where $a_i\in \F_{q^m}$, $t$ is a positive integer.
If $a_t \neq 0$, we say that $t=\deg_q f(x)$ is the $q$-\emph{degree} of $f$.
We denote by $\mathcal{L}_{m,q}$ the set of all $q$-polynomials over $\F_{q^m}$ and by $\tilde{\mathcal{L}}_{m,q}$ the following quotient $\mathcal{L}_{m,q}/(x^{q^m}-x)$.
The $\F_q$-linear maps of $\F_{q^m}$ can be identified with the polynomials in $\tilde{\mathcal{L}}_{m,q}$.
This shows the relevance of linearized polynomials in the theory of finite fields and their algebraic and geometric applications.
A fundamental problem in the theory of linearized polynomials is to characterize precisely the dimension of the kernel of the given polynomial in terms of its coefficients.
Results in this direction are given in \cite{qres,teoremone,GQ2009,McGuireSheekey,PZ2019,wl,Zanella}.

Let $n,s$ be positive integers such that $s<2n$ $\gcd(s,n)=1$.
First in \cite{CMPZ}, and later in \cite{PZ2019}, the following polynomials are investigated
\begin{equation}\label{eq:form} f_{a,b,s}(x)=x+ax^{q^s}+bx^{q^{s+n}} \in \tilde{\cL}_{2n,q}.
\end{equation}
The following results are known from \cite{CMPZ} and \cite{PZ2019}:
\begin{itemize}
    \item if $\N_{q^{2n}/q^n}(a)=\N_{q^{2n}/q^n}(b)$, then $\dim_{\F_q} \ker f_{a,b,s}(x)\leq 1$;
    \item if $\N_{q^{2n}/q^n}(a)\neq \N_{q^{2n}/q^n}(b)$, then $\dim_{\F_q} \ker f_{a,b,s}(x)\leq 2$;
\end{itemize}
where $\N_{q^{2n}/q^n}(x)=x^{1+q^n}$.

Our main result is Theorem \ref{th:mainmain} and concerns the existence, for every $\delta\in\F_{q^{2n}}$ with $\mathrm{N}_{q^{2n}/q^n}(\delta)\notin\{0,1\}$, of an element $a\in\F_{q^{2n}}$ such that the kernel of $f_{a,\delta a,s}$ has dimension $2$, providing $n$ is large enough.

\begin{theorem}\label{th:mainmain}
Let $q$ be a prime power and $n,s$ be two relatively prime positive integers. Suppose that
\[
n\geq\begin{cases} 4s+2 & \textrm{if}\; q=3\textrm{ and }s>1,\,\textrm{or}\;q=2\textrm{ and }s>2; \\ 4s+1 & \textrm{otherwise}. \end{cases}
\]
For every $\delta \in \F_{q^{2n}}^*$ with $\N_{q^{2n}/q^n}(\delta)\neq 1$ there exists $a \in \F_{q^{2n}}^*$ such that
\[ \dim_{\F_q} \ker (f_{a,b,s}(x))=2, \]
where $b=\delta a$.
\end{theorem}
In Remark \ref{rem:adjoint} we show that we can always suppose $n>2s$, up to considering the adjoint polynomial.

The first step in the proof of Theorem \ref{th:mainmain} is to manipulate the shape of $f_{a,b,s}(x)$ to translate the condition on the dimension of the kernel into the existence of $\mathbb{F}_{q^n}$-rational points in the intersection of certain $\F_{q^n}$-rational hypersurfaces, which are described in Theorem \ref{th:main}.
Then we prove that this intersection is described by means of an $\F_{q^{2n}}$-rational curve $\mathcal{X}$. Using intersection theory and function field theory, the curve $\mathcal{X}$ is shown to be absolutely irreducible of genus $q^{2s}-q^s-1$; Theorem \ref{th:mainmain} now follows by Hasse-Weil bound.

Theorem \ref{th:mainmain} also has applications in the theory of scattered polynomials. A polynomial $f(x)\in\tilde{\cL}_{m,q}$ is said to be \emph{scattered} if
\[
\dim_{\F_q}\ker(f(x)-\lambda x)\leq 1, \quad\textrm{for all}\;\; \lambda\in\F_{q^m}.
\]
Scattered polynomials have been widely investigated, especially after the paper \cite{Sheekey2016}, where Sheekey builds a bridge between scattered polynomials and rank metric codes.
The family of linearized binomials $f_{\delta,s}(x)=x^{q^s}+\delta x^{q^{n+s}}\in\tilde{\cL}_{2n,q}$ with $\delta\ne0$ contains a large number of scattered polynomials when $n$ is $3$ or $4$, as proved in \cite{CMPZ} and \cite{PZ2019}.
The question arises whether there exist other values of $n$, possibly infinitely many, for which $f_{\delta,s}(x)$ is scattered.
Many authors have considered the problem of classifying \emph{exceptional} scattered polynomials $f(x)\in\tilde{\cL}_{m,q}$, i.e. scattered polynomials which remain scattered over infinitely many extensions $\F_{q^{\ell m}}$ of $\F_{q^m}$; partial classification results have been provided by Bartoli and Zhou \cite{BZ}, Bartoli and Montanucci \cite{BM}, Ferraguti and Micheli \cite{FM}.
Their results rely on the fact that the order of $\F_{q^{\ell m}}$ is much larger than the degree of $f(x)$; as a matter of fact, the key role in \cite{BZ,BM} is played by the application of the Hasse-Weil bound to a curve whose degree has the same order of magnitude as $\deg f(x)$, and hence is small with respect to $q^{\ell m}$ (see \cite[Lemma 2.1]{BZ}).
The aforementioned binomial $f_{\delta,s}(x)$ is not taken into account by their results, because $\deg f_{\delta,s}(x)=q^{n+s}$ is high with respect to the order of $\F_{q^{2n}}$.
As a byproduct of Theorem \ref{th:mainmain}, we prove in Theorem \ref{th:noscatt} that $f_{\delta,s}(x)$ is not scattered when $n$ is large enough with respect to $s$; for instance, when $s=1$ it is enough to choose $n\geq5$.

Finally, in Theorem \ref{th:applMRD} we use Theorem \ref{th:mainmain} to give an asymptotic classification of the family of rank-metric codes defined by the binomials $f_{\delta,s}(x)$.

The paper is organized as follows.
Section \ref{sec:preliminaries} contains preliminary results about algebraic curves and function function fields which are used in Section \ref{sec:proof}.
Section \ref{sec:proof} is devoted to the proof of Theorem \ref{th:mainmain}; the cases $q$ odd and $q$ even are studied separately, respectively in Section \ref{sec:qodd} and in Section \ref{sec:qeven}.
Section \ref{sec:appl} provides the applications of Theorem \ref{th:mainmain}; namely, Section \ref{sec:linearsets} shows the applications to scattered polynomials and linear sets, while Section \ref{sec:MRD} shows the applications to rank metric codes.

\section{Preliminaries on algebraic curves}\label{sec:preliminaries}

Let $\cC$ be a projective, absolutely irreducible, algebraic curve over the algebraically closed field $\mathbb{K}=\overline{\mathbb{F}}_q$, embedded in a projective space $\PG(r,\mathbb{K})$ with homogeneous coordinates $(X_1\colon\ldots\colon X_{r+1})$ and not contained in the hyperplane at infinity $H_{\infty}:X_{r+1}=0$.
Let $I(\mathcal{C})$ be the ideal of $\mathcal{C}$.
Denote by $\mathbb{K}(\mathcal{C})$ the field of ($\mathbb{K}$-)rational functions on $\mathcal{C}$, briefly the function field of $\mathcal{C}$. Clearly, $\mathbb{K}(\mathcal{C})$ is generated over $\mathbb{K}$ by the coordinate functions $x_1,\ldots,x_r$ with $x_i=\frac{X_i+I(\C)}{X_{r+1}+I(\C)}$, and $\mathbb{K}(\C)\colon\mathbb{K}$ is a field extension of transcendence degree $1$.
We denote by $\mathbb{P}(\mathcal{C})$ the set of places of $\mathcal{C}$, that is, the set of places of its function field $\mathbb{K}(\mathcal{C})$.
For every $P\in\mathbb{P}(\mathcal{C})$ and every nonzero $z\in\mathbb{K}(\C)$, we denote by $v_P(z)\in\mathbb{Z}$ the valuation of $z$ at $P$; $P$ is said to be a zero (resp. a pole) of $z$ if $v_P(z)>0$ (resp. $v_P(z)<0$).

Suppose that $\C$ is defined over $\mathbb{F}_q$, i.e. $I(\C)$ is generated by polynomials over $\mathbb{F}_q$.
Then $\mathbb{F}_q(\C)$ denotes the $\mathbb{F}_q$-rational function field of $\C$, i.e. the field of $\mathbb{F}_q$-rational functions on $\mathcal{C}$.
The $\mathbb{F}_q$-rational places of $\C$ are those places $P\in\mathbb{P}(\C)$ which are defined over $\mathbb{F}_q$; that is, $\mathbb{F}_q$-rational places of $\C$ are the places of degree $1$ in $\mathbb{F}_q(\C)$, which are exactly the restriction to $\mathbb{F}_q(\C)$ of the places of $\mathbb{K}(\C)$ in the constant field extension $\mathbb{K}(\C)\colon\mathbb{F}_q(\C)$.
The center of an $\mathbb{F}_q$-rational place is an $\mathbb{F}_q$-rational point of $\cC$; conversely, if $P$ is a simple $\mathbb{F}_q$-rational point of $\cC$, then the only place centered at $P$ is $\mathbb{F}_q$-rational, and may be identified with $P$.

Let $\varphi:\C^\prime\to\C$ be a covering of curves, i.e. a non-constant rational map from the curve $\C^\prime$ to the curve $\C$, of degree $\deg(\varphi)=[\mathbb{K}(\C^\prime)\colon\mathbb{K}(\C)]$. We denote by $\varphi$ also the induced map $\mathbb{P}(\C^\prime)\to\mathbb{P}(\C)$; if $\varphi$ is $\mathbb{F}_q$-rational, then $\varphi$ maps $\mathbb{F}_q$-rational places of $\C^\prime$ to $\mathbb{F}_q$-rational places of $\C$. The pull-back of $\varphi$ is denoted by $\varphi^*:\mathbb{K}(\C)\to\mathbb{K}(\C^\prime)$.
When $P\in\mathbb{P}(\C)$ and $P^\prime\in\mathbb{P}(\C^\prime)$ satisfy $\varphi(P^\prime)=P$, we write $P^\prime|P$ and say that $P^\prime$ lies over $P$ in $\varphi$.
We denote by $e(P^\prime|P)$ the ramification index of $P^\prime|P$, that is the unique positive integer such that $v_{P^\prime}(\varphi^*(w))=e(P^\prime|P)\cdot v_P(w)$ for all $w\in\mathbb{K}(\C)$; we have $\sum_{P^\prime:P^\prime|P}e(P^\prime|P)=\deg(\varphi)$.
We say that $P^\prime$ is ramified over $P$ if $e(P^\prime|P)>1$, and totally ramified if $e(P^\prime|P)=\deg(\varphi)$; otherwise it is unramified.
A ramified place $P^\prime$ is wildly ramified (resp. tamely ramified) if $e(P^\prime|P)$ is divisible (resp. not divisible) by $p$.
We refer to \cite{HKT,Sti} for further details on algebraic curves and function fields.

\begin{theorem}\label{th:hurwitz}{\rm (Hurwitz genus formula, \cite[Theorem 3.4.13]{Sti})}
Let $\mathcal{C},\mathcal{C}^\prime$ be two absolutely irreducible curves over $\mathbb{K}=\overline{\mathbb{F}}_q$ and $\varphi:\mathcal{C}^\prime\to\mathcal{C}$ be a covering.
For every place $P$ of $\mathcal{C}$ and every place $P^\prime$ of $\mathcal{C}^\prime$ lying over $P$ in $\varphi$, let $t\in\mathbb{K}(\mathcal{C})$ be a local parameter at $P$, $t^\prime\in\mathbb{K}(\mathcal{C}^\prime)$ be a local parameter at $P^\prime$, and $\varphi^*(t)\in\mathbb{K}(\mathcal{C}^\prime)$ be the pull-back of $t$ with respect to $\varphi$.
Then
\[ 2g(\mathcal{C}^\prime)-2=\deg(\varphi)\cdot(2g(\mathcal{C})-2)+\sum_{P^\prime\in\mathbb{P}(\mathcal{C}^\prime)}v_{P^\prime}\left(\frac{d\varphi^*(t)}{d t^\prime}\right). \]
\end{theorem}

If $P^\prime$ is not wildly ramified, then $v_{P^\prime}\left(\frac{d\varphi^*(t)}{d t^\prime}\right)=e(P^\prime|P)-1$.
We now recall two important types of coverings.
The following results are the application of \cite[Corollary 3.7.4]{Sti} and \cite[Theorem 3.7.10]{Sti} in the case of an algebraically closed constant field $\mathbb{K}$.

\begin{theorem}\label{th:kummer}{\rm \cite[Corollary 3.7.4]{Sti}}
Let $\C\colon F(X,Y)=0$ be an absolutely irreducible plane curve defined over a finite field $\mathbb{F}_q$ of characteristic $p$, and $m$ be a positive integer with $\gcd(m,p)=1$.
Let $f(X,Y)\in\mathbb{F}_q[X,Y]$ be such that there exists an $\overline{\mathbb F}_q$-rational place $Q$ of $\mathcal{C}$ at which the valuation of the rational function $f(x,y)$ is coprime with $m$, i.e. $\gcd(v_Q(f(x,y)),m)=1$.
Let $\mathcal{C}^\prime$ be the curve given by the two affine equations $F(X,Y)=0$ and $Z^m=f(X,Y)$.
Then the following holds.
\begin{itemize}
    \item $\mathcal{C}^\prime$ is absolutely irreducible and defined over $\mathbb{F}_q$; $\C'$ is called a \emph{Kummer cover} of $\mathcal{C}$.
    \item The $\mathbb{F}_q$-rational covering $\varphi:\C^\prime\to\C$, $(X,Y,Z)\mapsto(X,Y)$, has degree $m$.
    \item For every place $P$ of $\mathcal{C}$ and every place $P^\prime$ of $\C'$ lying over $P$ in $\varphi$, we have $e(P^\prime| P)=m/r_P$, where $r_P=\gcd(v_P(f(x,y)),m)>0$.
    \item The Hurwitz genus formula reads
    \[ g(\C')=1+m(g(\C)-1)+\frac{1}{2}\sum_{P\in\mathbb{P}(\mathcal{C})}(m-r_P). \]
\end{itemize}
\end{theorem}

If $\C^\prime$ is an absolutely irreducible curve over $\mathbb{F}_q$ defined by the two affine equations $F(X,Y)=0$ and $L(Z)=f(X,Y)$, for some $f(X,Y),F(X,Y)\in\mathbb{F}_q[X,Y]$ and some separable $p$-polynomial $L(T)\in\mathbb{F}_q[T]$, then $\mathcal{C}^\prime$ is said to be a \emph{generalized Artin-Schreier cover} of the curve $\mathcal{C}:F(X,Y)=0$, with generalized Artin-Schreier covering $\varphi:\C^\prime\to\C$, $(X,Y,Z)\mapsto(X,Y)$.

\begin{theorem}\label{th:artinschreier}{\rm \cite[Theorem 3.7.10]{Sti}}
Let $\mathcal{C}:F(X,Y)=0$ be an absolutely irreducible plane curve defined over a finite field $\mathbb{F}_q$ of charateristic $p$.
Let $L(T)\in\mathbb{F}_q[T]$ be a separable $p$-polynomial of degree $\bar{q}$ with all its roots in $\mathbb{F}_q$.
Let $f(X,Y)\in\mathbb{F}_q[X,Y]$ be such that for every place $P\in\mathbb{P}(\mathcal{C})$ there exists a rational function $\omega$ on $\mathcal{C}$ (depending on $P$) satisfying either $v_P(f(x,y)-L(\omega))\geq0$ or $v_P(f(x,y)-L(\omega))=-m$ with $m>0$ and $p\nmid m$.
Define $m_P=-1$ in the former case and $m_P=m$ in the latter case.
Let $\mathcal{C}^\prime$ be the space curve given by the two affine equations $F(X,Y)=0$ and $L(Z)=f(X,Y)$.
If there exists a place $Q\in\mathbb{P}(\mathcal{C})$ with $m_Q>0$, then $\mathcal{C}^\prime$ is a generalized Artin-Schreir cover of $\C$, defined over $\mathbb{F}_q$.

With the above notation, the following holds for generalized Artin-Schreier curves.
\begin{itemize}
    \item The $\mathbb{F}_q$-rational covering $\varphi:\C^\prime\to\C$, $(X,Y,Z)\mapsto(X,Y)$, has degree $\bar{q}$.
    \item For every place $P$ of $\mathcal{C}$ and every place $P^\prime$ of $\mathcal{C}^\prime$ lying over $P$ in $\varphi$, $e(P^\prime|P)$ is equal either to $1$ or to $\bar{q}$ according to $m_P=-1$ or $m_P>0$, respectively.
    \item The Hurwitz genus formula reads
    \[ g(\mathcal{C}^\prime)=\bar{q}\cdot g(\mathcal{C})+\frac{\bar{q}-1}{2}\cdot\left(-2+\sum_{P\in\mathbb{P}(\mathcal{C})}(m_P+1)\right). \]
\end{itemize}
\end{theorem}

We now recall the well-known Hasse-Weil bound.

\begin{theorem}\label{th:hasseweil}{\rm \cite[Theorem 5.2.3]{Sti} (Hasse-Weil bound)}
Let $\mathcal{C}$ be an absolutely irreducible curve defined over $\mathbb{F}_q$ and with genus $g$. Then the number $N_{q}$ of $\mathbb{F}_q$-rational places of $\mathcal{C}$ satisfies
\[q+1-2g\sqrt{q}\leq N_{q} \leq q+1+2g\sqrt{q}.\]
\end{theorem}



\section{Proof of Theorem \ref{th:mainmain}}\label{sec:proof}

In this section we prove Theorem \ref{th:mainmain}. First we determine necessary and sufficient conditions on $a$ and $b$ for $f_{a,b,s}(x)$ having kernel of dimension $2$; cf. Theorem \ref{th:main}. Then we investigate such conditions by means of algebraic-geometric tools.

The first remark shows that different choices of $a,b$ with the same norm of $b/a$ over $\mathbb{F}_{q^n}$ provide polynomials $f_{a,b,s}(x)$ with the same behaviour.

\begin{remark}\label{rk:normdelta}
Assume that the linearized polynomial $f_{a,b,s}(x)=x+ax^{q^s}+bx^{q^{s+n}} \in \F_{q^{2n}}[x]$,
with $\gcd(s,n)=1$ and $b=\delta a$, has kernel of dimension two.
Clearly, for each $\lambda \in \F_{q^{2n}}^*$ we have
\[ \dim_{\F_q} \ker(\lambda^{-1}f_{a,b,s}(\lambda x))=2, \]
where
\[\lambda^{-1}f_{a,b,s}(\lambda x)=x+a \lambda^{q^s-1}x^{q^s}+a\lambda^{q^s-1}\delta \lambda^{q^s(q^n-1)}x^{q^{s+n}}=f_{a',b',s}(x),\]
with $a'=a\lambda^{q^s-1}$, $\delta'=\lambda^{q^s(q^n-1)}\delta$ and $b'=a'\delta'$.
Note that for each element $\delta' \in \F_{q^{2n}}$ with $\N_{q^{2n}/q^n}(\delta')=\N_{q^{2n}/q^n}(\delta)$ there exists $\lambda \in \F_{q^{2n}}$ such that $\delta'=\delta \lambda^{q^s(q^n-1)}$.
Therefore, if $\dim_{\F_q} \ker(f_{a,b,s}(x))=2$, with $b=\delta a$, then for each $\delta' \in \F_{q^{2n}}$ with $\N_{q^{2n}/q^n}(\delta')=\N_{q^{2n}/q^n}(\delta)$ there exists $a' \in \F_{q^{2n}}$ such that $\dim_{\F_q} \ker(f_{a',b',s}(x))=2$, with $b'=\delta' a'$.
\end{remark}

The second remark shows that we may assume $s<n/2$.

\begin{remark}\label{rem:adjoint}
The \emph{adjoint} of a $q$-polynomial $f(x)=\sum_{i=0}^{n-1}a_i x^{q^i}$, with respect to the bilinear form $\langle x,y\rangle=\mathrm{Tr}_{q^n/q}(xy)$, is given by
\[\hat{f}(x)=\sum_{i=0}^{n-1}a_{i}^{q^{n-i}} x^{q^{n-i}}.\]
In particular, if $f(x)$ is a $q$-polynomial of shape \eqref{eq:form}, then
\[ f_{a,b,s}(x)=x+ax^{q^s}+bx^{q^{n+s}}\in \tilde{\mathcal{L}}_{2n,q}, \]
with $\gcd(s,n)=1$ and its adjoint is
\[ \hat{f}_{a,b,s}(x)=x+a^{q^{2n-s}}x^{q^{2n-s}}+b^{q^{n-s}}x^{q^{n-s}}. \]
Therefore, choosing $s^\prime=2n-s$, $a^\prime=a^{q^{2n-s}}$, $b^\prime=b^{q^{n-s}}$, we get
\[ \hat{f}_{a,b,s}(x)=f_{a',b',s'}(x), \]
while choosing $s^{\prime\prime}=n-s$, $a^{\prime\prime}=b^{q^{n-s}}$, $b^{\prime\prime}=a^{q^{2n-s}}$,  we get
\[ \hat{f}_{a,b,s}(x)=f_{a^{\prime\prime},b^{\prime\prime},s^{\prime\prime}}(x), \]
i.e. $\hat{f}_{a,b,s}(x)$ is of shape \eqref{eq:form}.
Therefore, the family of $q$-polynomials we are studying is closed by the adjoint operation.
Furthermore, we underline that by \cite[Lemma 2.6]{BGMP2015}, the kernels of $f_{a,b,s}$ and $\hat{f}_{a,b,s}$ have the same dimension (see also \cite[pages 407--408]{CsMP}).
Thus, we can assume $s< n/2$.
\end{remark}

We now prove that the shape of $\delta$ can be chosen as in \eqref{eq:deltachoice}.

\begin{theorem}\label{th:deltachoice}
Let $f_{a,b,s}(x) \in \F_{q^{2n}}[x]$, with $b=a\delta$.
Then $\dim_{\F_q}\ker(f_{a,b,s}(x))=2$ if and only if $\dim_{\F_q}\ker(f_{\overline{a},\overline{b},s}(x))=2$, with
\begin{equation}\label{eq:deltachoice}
\overline{\delta}=\frac{\xi^{q^{s+n}}-\xi^{q^n}}{\xi^{q^n}-\xi^{q^s}},
\end{equation}
for some $\xi \in \F_{q^{2n}}\setminus\F_{q^n}$ and some $\overline{a}\in \F_{q^{2n}}$, $\overline{b}=\overline{\delta}\overline{a}$.
\end{theorem}

\begin{proof}
Assume that $\dim_{\F_q}\ker(f_{a,b,s}(x))=2$, i.e. there exist $x_0 \in \F_{q^{2n}}^*$ and $y_0\in \F_{q^{2n}}\setminus\F_q$ such that $x_0/y_0 \notin \F_q$ and
\[ \frac{x_0^{q^s}+\delta x_0^{q^{s+n}}}{x_0}=\frac{y_0^{q^s}+\delta y_0^{q^{s+n}}}{y_0}, \]
which may be rewritten as follows
\[ \delta(y_0 x_0^{q^{s+n}}-x_0y_0^{q^{s+n}})=x_0y_0^{q^s}-y_0x_0^{q^s}. \]
If $y_0 x_0^{q^{s+n}}-x_0y_0^{q^{s+n}}$ would be zero, than $x_0/y_0 \in \F_{q^{2n}}\cap\F_{q^{s+n}}=\F_q$, a contradiction.
Hence,
\[ \delta=\frac{x_0y_0^{q^s}-y_0x_0^{q^s}}{y_0 x_0^{q^{s+n}}-x_0y_0^{q^{s+n}}}, \]
and, since $y_o=\xi x_0$ for some $\xi \in \F_{q^{2n}}\setminus \F_q$, we have
\[ \delta= \frac{1}{-x_0^{q^{s+n}-q^s}} \frac{\xi^{q^s}-\xi}{\xi^{q^{s+n}}-\xi}. \]
By Remark \ref{rk:normdelta}, $\dim_{\F_q}\ker(f_{a,b,s}(x))=2$ if and only if there exists $\overline{a},\overline{b}$ as in the claim such that $\dim_{\F_q}\ker(f_{\overline{a},\overline{b},s}(x))=2$.
If $\xi\in\mathbb{F}_{q^n}$, then $\overline{\delta}=-1$, and hence $\dim_{\F_q}\ker(f_{\overline{a},\overline{b},s}(x))\leq1$.
The claim follows.
\end{proof}

As a consequence of Theorem \ref{th:deltachoice} we get the following result.

\begin{corollary}
There exist $\delta \in \F_{q^{2n}}^*$ for which $\dim_{\F_q}\ker(f_{a,b,s}(x))\leq 1$, with $b=\delta a$, for each $a\in \F_{q^{2n}}^*$ if and only if
\[ \left| \left\{ \N_{q^{2n}/q^n}\left( \frac{\xi^{q^{n+s}}-\xi^{q^n}}{\xi^{q^n}-\xi^{q^s}} \right) \colon \xi \in \F_{q^{2n}}\setminus\F_{q^n} \right\} \right|< q^n-1. \]
\end{corollary}

Since $\xi \notin \F_{q^n}$, we have that $\xi$ is the root of an irreducible polynomial $X^2-SX-T \in \F_{q^n}[X]$, where $\N_{q^{2n}/q^n}(\xi)=-T$ and $\mathrm{Tr}_{q^{2n}/q^n}(\xi)=S$.
Also, $\{1,\xi\}$ is an $\F_{q^n}$-basis of $\F_{q^{2n}}$ and so there exist $A,B \in \F_{q^n}$ such that $\xi^{q^s}=A+B\xi$.
In the next we give some relations involving $A,B,S$ and $T$.

\begin{proposition}
The following holds:
\begin{enumerate}
    \item $S^{q^s}=2A+BS$;
    \item $-T^{q^s}=A^2+B(AS-BT)$.
\end{enumerate}
In particular, $\mathrm{Tr}_{q^{2n}/q^n}(\xi^{q^s+1})=2BT+AS+BS^2$ and $\mathrm{Tr}_{q^{2n}/q^n}(\xi^{q^s+q^n})=AS-2BT$.
\end{proposition}

\begin{proof}
As
\[ \xi^{q^s+q^n}=(A+B\xi)(S-\xi)=AS-BT-A\xi \]
and
\[ \xi^{1+q^{n+s}}=-BT+(S^{q^s}-A-BS)\xi, \]
we have that
\[ \mathrm{Tr}_{q^{2n}/q^n}(\xi^{q^s+q^n})=AS-2BT+(S^{q^s}-2A-BS)\xi. \]
Since $\mathrm{Tr}_{q^{2n}/q^n}(\xi^{q^s+q^n})\in \F_{q^n}$, we get the first relation.
Also,
\[ -T^{q^s}=\N_{q^{2n}/q^n}(\xi^{q^s})=A^2+ABS-B^2T, \]
i.e. the second relation.
\end{proof}

Let $\alpha \in \F_{q^n}^*$ with $\alpha \neq 1$. Then
\[ \N_{q^{2n}/q^n}\left( \frac{\xi^{q^{n+s}}-\xi^{q^n}}{\xi^{q^n}-\xi^{q^s}} \right)=\frac{\xi^{q^n+1}+\xi^{q^s+q^{n+s}}-(\xi^{1+q^{n+s}}+\xi^{q^s+q^n})}{\xi^{q^n+1}+\xi^{q^s+q^{n+s}}-(\xi^{q^n+q^{n+s}}+\xi^{q^s+1})}=\alpha, \]
which can be written as
\[ (1-\alpha)(T+T^{q^s})-\alpha S^{q^s+1}+(1+\alpha)(AS-2BT)=0. \]

Hence, we have the following result.

\begin{theorem}\label{th:main}
Let $\alpha \in \F_{q^n}^*$ with $\alpha \ne 1$ and $s$ a positive integer with $\gcd(s,n)=1$. If there exist $T,S,A,B \in \F_{q^n}$ such that
\begin{enumerate}
    \item $(1-\alpha)(T+T^{q^s})-\alpha S^{q^s+1}+(1+\alpha)(AS-2BT)=0$;
    \item $X^2-SX-T \in \F_{q^n}[X]$ is irreducible over $\F_{q^n}$;
    \item $S^{q^s}=2A+BS$;
    \item $-T^{q^s}=A^2+B(AS-BT)$,
\end{enumerate}
then for every $\delta\in \F_{q^{2n}}$, with $\N_{q^{2n}/q^n}(\delta)=\alpha$, there exists $a\in\mathbb{F}_{q^{2n}}^*$ such that $\dim_{\mathbb{F}_q}\ker(f_{a,b,s}(x))=2$, where $b=\delta a$.
\end{theorem}

In the rest of this section $q=p^h$ with $p$ prime.
We will show that the existence of the parameters $T,S,A,B\in\mathbb{F}_{q^n}$ satisfying the hypothesis of Theorem \ref{th:main} is equivalent to the existence of a suitable affine $\mathbb{F}_{q^n}$-rational point of the algebraic plane curve with equation \eqref{eq:curveqodd} or \eqref{eq:curveqeven}, for $q$ odd or $q$ even respectively.

\subsection{Proof of Theorem \ref{th:mainmain} for $q$ odd}\label{sec:qodd}

Denote by $\Delta=S^2+4T$.
By 3. and 4. of Theorem \ref{th:main}, we get
\[ B=\epsilon\Delta^{\frac{q^s-1}{2}},\quad A=\frac{1}{2}(S^{q^s}-\epsilon S \Delta^{\frac{q^s-1}{2}}), \]
where $\epsilon\in\{1,-1\}$.
Hence we get $AS-2BT=\frac{1}{2}\epsilon\Delta^{\frac{q^s-1}{2}}-\frac{1}{2}S^{q^s+1}$. Replacing such values in 1. of Theorem \ref{th:main}, we get
\begin{equation}\label{eq:qodd1}
2(T+T^{q^s})(1-\alpha)+(1-\alpha)S^{q^s+1}=\epsilon(\alpha+1)\Delta^{\frac{q^s+1}2}.
\end{equation}
Also, the irreducibility of $X^2-SX-T$ over $\F_{q^n}$ is equivalent to the existence of a nonsquare element $\eta$ of $\F_{q^n}$ and a nonzero element $Z$ of $\F_{q^n}$ such that $\Delta=\eta Z^2$.
Therefore, \eqref{eq:qodd1} becomes
\[ 2(T+T^{q^s})+S^{q^s+1}=\beta \eta^{\frac{q^s+1}2} Z^{q^s+1}, \]
where $\beta=\epsilon\frac{\alpha+1}{1-\alpha}$.
Using that $T=\frac{\eta Z^2-S^2}{4}$, we get the following equation:
\begin{equation}\label{eq:curveqodd} -(S^{q^s}-S)^2+\eta Z^2 + \eta^{q^s} Z^{2q^s} - 2\beta\eta^{\frac{q^s+1}{2}}Z^{q^s+1}=0.
\end{equation}

\begin{theorem}\label{th:qodd}
Let $\beta\in\mathbb{F}_{q^n}\setminus\{1,-1\}$ and $\eta$ a non-square in $\mathbb{F}_{q^n}$.
The plane curve $\mathcal{C}$ with affine equation \eqref{eq:curveqodd} is absolutely irreducible and has genus $g(\mathcal{C})=q^{2s}-q^s-1$.
\end{theorem}

\begin{proof}
Let $G(Z)=\eta Z^2+\eta^{q^s} Z^{2q^s}-2\beta\eta^{\frac{q^s+1}{2}}Z^{q^s+1}\in\mathbb{F}_{q^n}[Z]$, and let $\mathcal{C}_1$ be the plane curve with affine equation $F_1(U,Z)=0$, where $F_1(U,Z)=U^2-G(Z)$.
By direct computation using the assumption $\beta\ne\pm1$ follows that $0$ is the unique multiple root of $G(Z)$, with multiplicity $2$; the other $2q^s-2$ roots $\lambda_1,\ldots,\lambda_{2q^s-2}$ of $G(Z)$ are simple.
Then $G(Z)$ is not a square in $\mathbb{K}[Z]$, whence $F_1(U,Z)$ is irreducible over $\mathbb{K}=\overline{\mathbb{F}}_{q^n}$, i.e. $\mathcal{C}_1$ is absolutely irreducible.

The genus of the quadratic Kummer cover $\mathcal{C}_1$ of the projective line is computed as follows.
Let $z,u$ be the coordinate functions of $\mathcal{C}_1$, so that the function field of $\mathcal{C}_1$ is $\mathbb{K}(\mathcal{C}_1)=\mathbb{K}(z,u)$.
The valuation of $G(z)$ at the zero of $z-\lambda_i$ in $\mathbb{K}(\mathbb{P}_z^1)=\mathbb{K}(z)$ is $1$, for every $i=1,\ldots,2q^s-2$. The valuation of $G(z)$ at any other place of $\mathbb{P}_z^1$ is even; namely, it is $2$ at the zero of $z$, $-2q^s$ at the pole of $z$, and $0$ at the zero of $z-\mu$ whenever $G(\mu)\ne0$.
By Theorem \ref{th:kummer}, the only ramified places in $\mathcal{C}_1\to\mathbb{P}_z^1$ are the zeros of $z-\lambda_1,\ldots,z-\lambda_{2q^s-2}$; hence,
\[g(\mathcal{C}_1)= 1+2(g(\mathbb{P}_z^1)-1)+\frac{1}{2}(2q^s-2)(2-1)=q^s-2.\]


Since $\mathcal{C}$ has equation $(S^{q^s}-S)^2=G(Z)$, it is enough to show that $\mathcal{C}$ is an Artin-Schreier cover of $\mathcal{C}_1$, with covering $\varphi:\mathcal{C}\to\mathcal{C}_1$, $(Z,S)\mapsto(Z,U=S^{q^s}-S)$, of degree $q^s$.
To this aim, consider the two poles $P_{\infty}$ and $Q_\infty$ of $u$ on $\mathcal{C}_1$; the rational function $1/z$ is a local parameter at each of them, i.e. $v_{P_{\infty}}(1/z)=v_{Q_{\infty}}(1/z)=1$.
By direct computation, the Laurent series of $u$ at $P_\infty$ with respect to $1/z$ is \[u=\sqrt{\eta^{q^s}}\left(1/z\right)^{-q^s}-\beta\sqrt{\eta}\left(1/z\right)^{-1}+\frac{\eta-\beta^2\eta}{2\sqrt{\eta^{q^s}}}(1/z)^{q^s-2}+w,\]
for some $w\in\mathbb{K}(\mathcal{C}_1)$ with $v_{P_\infty}(w)>q^s-2$.
By choosing $\omega_{P_{\infty}}=\sqrt{\eta}z$ one has that $u-(\omega_{P_{\infty}}^{q^s}-\omega_{P_{\infty}})$ has valuation $-1$ at $P_\infty$, because $\beta\ne1$.
Analogously, there exists $\omega_{Q_\infty}$ such that $v_{Q_{\infty}}(u-(\omega_{Q_{\infty}}^{q^s}-\omega_{Q_{\infty}}))=-1$.
Hence, by Theorem \ref{th:artinschreier}, $\mathcal{C}$ is an absolutely irreducible Artin-Schreier extension $\mathcal{C}_1$ of degree $q^s$.

The ramified places in $\mathcal{C}\to\mathcal{C}_1$ are exactly $P_\infty$ and $Q_\infty$, which are totally ramified; any other place of $\mathcal{C}_1$ is unramified under $\mathcal{C}$.
Therefore,
\[g(\mathcal{C})=q^s\cdot g(\mathcal{C}_1)+\frac{q^s-1}{2}\left(-2+2\cdot2\right)=q^{2s}-q^s-1.\]
\end{proof}

\begin{proposition}\label{prop:HWqodd}
Let $\mathcal{C}$ be the plane curve with affine equation \eqref{eq:curveqodd}.
If 
\[
n\geq\begin{cases} 4s+1 & \textrm{if }\,q>3, \\
4s+2 & \textrm{if }\,q=3,s>1,\\
5 & \textrm{if }\,q=3,s=1;
\end{cases}
\]
then there exists an $\mathbb{F}_{q^n}$-rational affine point $(\bar{z},\bar{s})$ of $\mathcal{C}$ such that $\bar{t}=\frac{\eta \bar{z}^2-\bar{s}^2}{4}$ is different from zero.
\end{proposition}

\begin{proof}
By Theorem \ref{th:qodd}, $\mathcal{C}$ is absolutely irreducible with genus $g(\mathcal{C})=q^{2s}-q^s-1$.
By Theorem \ref{th:hasseweil}, the number $N_{q^n}$ of $\mathbb{F}_{q^n}$-rational places of $\mathcal{C}$ satisfies
\[N_{q^n}\geq q^n+1-2(q^{2s}-q^s-1)\sqrt{q^n}.\]
From the proof of Theorem \ref{th:qodd} the following facts follow.
\begin{itemize}
    \item $z$ has exactly $2$ poles on $\mathcal{C}$, which coincide with the poles of $s$, namely the places lying over $P_\infty$ and $Q_\infty$.
    \item Using the equation of $\mathcal{C}$, the zeros of $t=\frac{\eta z^2-s^2}{4}=\frac{(\sqrt{\eta}z-s)(\sqrt{\eta}z+s)}{4}$ on $\mathcal{C}$ are also zeros of $(\beta-1)z^{q^s+1}$ and hence of $z$ as $\beta\ne1$; thus, they are the common zeros of $z$ and $s$ on $\mathcal{C}$, and there are exactly $2$ of them.
\end{itemize}
Altogether, there are $4$ places of $\mathcal{C}$ which are either poles of $s$ or $z$ or $t$, or zeros of $t$.
The assumption on $n$ implies that
\[ q^n+1-2(q^{2s}-q^s-1)\sqrt{q^n}>4, \]
whence $N_{q^n}>4$.
Then there exists an $\mathbb{F}_{q^n}$-rational place $P$ which is not a pole of $z$, $s$, or $t$, and is not a zero of $t$. Then the point $(\bar{z},\bar{s})=(z(P),s(P))$ yields the claim.
\end{proof}

From Theorem \ref{th:main} and Proposition \ref{prop:HWqodd} follows Corollary \ref{cor:mainmain_qodd}, which is our main result Theorem \ref{th:mainmain} when $q$ is odd.

\begin{corollary}\label{cor:mainmain_qodd}
Let $q$ be an odd prime power, $s\geq1$ be such that $\gcd(s,n)=1$.
Suppose that
\[
n\geq\begin{cases} 4s+2 & \textrm{if}\; q=3\textrm{ and }s>1; \\ 4s+1 & \textrm{otherwise}. \end{cases}
\]
Then for every $\delta\in\mathbb{F}_{q^{2n}}$ satisfying $\mathrm{N}_{q^{2n}/q^n}(\delta)\notin\{0,1\}$ there exists $a\in\mathbb{F}_{q^{2n}}^*$ such that $\dim_{\mathbb{F}_q}\ker(f_{a,b})=2$, where $b=\delta a$.
\end{corollary}

\subsection{Proof of Theorem \ref{th:mainmain} for $q$ even}\label{sec:qeven}

Let $q$ be a power of $2$. The conditions of Theorem \ref{th:main} read:
\begin{enumerate}
    \item $T+T^{q^s}+\beta S^{q^{s}+1}+AS=0$, with $\beta=\frac{\alpha}{1+\alpha}\notin\{0,1\}$;
    \item $S\ne0$ and $\mathrm{Tr}_{q^n/2}(T/S^2)=1$;
    \item $B=S^{q^s-1}$;
    \item $A^2+A S^{q^s}+S^{2q^s-2}T+T^{q^s}=0$.
\end{enumerate}
By 1. we get
\[ A= \beta S^{q^s}+\frac{T+T^{q^s}}S, \]
which can be replaced in 4. obtaining
\begin{equation}\label{eq:curveqeven} (\beta^2+\beta) S^{2(q^s+1)}+S^{q^s+1}(T^{q^s}+T)+S^{2q^s}T+S^2T^{q^s}+T^{2q^s}+T^2=0.
\end{equation}
Set $T=S^2 Y$. Then \eqref{eq:curveqeven} reads $H(S,Y)=0$, where
\begin{equation}\label{eq:miserve}
\begin{array}{lll}   H(S,Y)=Y^2+S^{4(q^s-1)}Y^{2q^s}+\beta^2 S^{2(q^s-1)}+S^{q^s-1}Y+ \\ \\
S^{3(q^s-1)}Y^{q^s}+\beta S^{2(q^s-1)}+S^{2(q^s-1)}Y+S^{2(q^s-1)}Y^{q^s}.
\end{array}
\end{equation}
Straightforward computation using $\mathrm{Tr}_{q^s/2}(Y)+\mathrm{Tr}_{q^s/2}(Y)^2=Y^{q^s}+Y$ shows that the polynomial $H(S,Y)$ in \eqref{eq:miserve} splits as follows.
\begin{lemma}
We have $H(S,Y)=G(S,Y)\cdot G^\prime(S,Y)$, where
\[ G(S,Y)= S^{2(q^s-1)}Y^{q^s}+S^{q^s-1}(1+\beta+ \mathrm{Tr}_{q^s/2}(Y))+Y, \]
\[ G^\prime(S,Y)= S^{2(q^s-1)}Y^{q^s}+S^{q^s-1}(\beta + \mathrm{Tr}_{q^s/2}(Y))+Y. \]
\end{lemma}

The condition 2. is equivalent to the existence of an element $Z\in \F_{q^n}$ such that
\begin{equation}\label{eq:T}
T=S^2(Z^2+Z+\epsilon),
\end{equation}
for some fixed $\epsilon\in\mathbb{F}_{q^n}$ such that $\mathrm{Tr}_{q^{n}/2}(\epsilon)=1$.

\noindent Let $\mathcal{C}$ be the plane curve with affine equation $F(S,Z)=G^\prime(S,Z^2+Z+\epsilon)$.

\medskip

In order to prove Theorem \ref{th:mainmain} when $q$ is even, by Theorem \ref{th:main} and the arguments at the beginning of Section \ref{sec:qeven}, it is enough to prove the existence of an $\fqn$-rational affine point $(\bar{s},\bar{z})$ of $\mathcal{C}$ such that $\bar{s}\ne0$ and $\bar{z}^2+\bar{z}+\epsilon\ne0$.
This is done by showing that $\mathcal{C}$ is absolutely irreducible, computing its genus, and applying the Hasse-Weil lower bound.
To this aim, we consider the following subcovers:
\[
\varphi_2:\C\to\C_2,\quad (S,Z)\mapsto(X=S^{q^s-1},Z), \]
\[ \varphi_1:\C_2\to\C_1,\quad (X,Z)\mapsto(X,Y=Z^2+Z+\epsilon).
\]
The curves $\C_2,\C_1$ have equation $\C_2\colon F_2(X,Z)=0$ and $\C_1\colon F_1(X,Y)=0$, where
\[ F_2(X,Z)= X^2(Z^2+Z+\epsilon)^{q^s}+X(\beta + \mathrm{Tr}_{q^s/2}(Z^2+Z+\epsilon))+Z^2+Z+\epsilon, \]
\[ F_1(X,Y)= X^2Y^{q^s}+X(\beta + \mathrm{Tr}_{q^s/2}(Y))+Y. \]

We first prove that $\mathcal{C}_2$ is absolutely irreducible by direct inspection, and that $\mathcal{C}$ is absolutely irreducible being a Kummer cover of $\mathcal{C}_2$. 
To compute the genus of $\mathcal{C}$, we start by the genus of the absolutely irreducible subcover $\mathcal{C}_1$ of $\mathcal{C}_2$, which is computed with the Hurwitz genus formula. 
Then we compute the genus of $\mathcal{C}_2$ as an Artin-Schreier cover of $\mathcal{C}_1$. 
Finally, the genus of the Kummer cover $\mathcal{C}$ of $\mathcal{C}_2$ is computed.

\medskip

Let $\gamma,\gamma+1$ be the roots of $Z^2+Z+\epsilon\in\mathbb{F}_{q^n}[Z]$. Hence, $\mathrm{Tr}_{q^{2n}/q^n}(\gamma)=1$ and $\mathrm{N}_{q^{2n}/q^n}(\gamma)=\epsilon$; also, $\mathrm{Tr}_{q^{n}/2}(\epsilon)=1$ implies $\gamma\in\mathbb{F}_{q^{2n}}\setminus\mathbb{F}_{q^n}$.

 \begin{lemma}\label{lemma:C'2irr}
 The curve $\mathcal{C}_2$ is absolutely irreducible.
 \end{lemma}
 \begin{proof}
By contradiction, suppose $F_2(X,Z)=\hat{F}(X,Z)\cdot \tilde{F}(X,Z)$ for some non-constant polynomials $\hat{F},\tilde{F}\in\mathbb{K}[X,Z]$.
Then $\hat{F}(X,Z)=X\cdot A(Z)+B(Z)$ and $\tilde{F}(X,Z)=X\cdot C(Z)+D(Z)$; up to scalar multiplication,
\[
A(Z)=(Z+\gamma)^a(Z+\gamma+1)^b,\qquad C(Z)=(Z+\gamma)^c(Z+\gamma+1)^d,
\]
where $a,b,c,d\geq0$ satisfy $a+c=b+d=q^s$.
Also,
\[ (Z+\gamma)^a(Z+\gamma+1)^bD(Z)+(Z+\gamma)^c(Z+\gamma+1)^dB(Z)=\beta + \mathrm{Tr}_{q^s/2}(Z^2+Z+\epsilon). \]
Clearly, $a,b \in \{0,q^s\}$ as $\beta\ne0$.
If $a=0$, then $b=c=q^s$ and $d=0$, since $B(Z)D(Z)=Z^2+Z+\epsilon$; hence
\begin{equation}\label{eq:Zirr} (Z+\gamma+1)^{q^s}D(Z)+(Z+\gamma)^{q^s}B(Z)=\beta + \mathrm{Tr}_{q^s/2}(Z^2+Z+\epsilon).
\end{equation}
This implies $D(\gamma)\ne0$ and $B(\gamma+1)\ne0$, whence $D(Z)=\lambda(Z+\gamma+1)$ and $B(Z)=\lambda^{-1}(Z+\gamma)$ for some $\lambda\in\mathbb{K}^*$. With $Z=\gamma$ in \eqref{eq:Zirr}, we get $\beta=1$, a contradiciton.
If $a=q^s$, the same arguments yield a contradiction.
 \end{proof}

\begin{proposition}\label{prop:C'1}
The curve $\C_1$ is absolutely irreducible with genus $q^s/2$.
\end{proposition}

\begin{proof}
As $\C_1$ is a subcover of $\C_2$, the absolute irreducibility of $\C_1$ follows from Lemma \ref{lemma:C'2irr}.
Let $(X\colon Y\colon V)$ be the homogeneous coordinates of the affine point $(X,Y)$.
By direct computation, the affine points of $\mathcal{C}_1$ are simple, and hence we identify each of them with the unique place of $\mathcal{C}_1$ centered at it; the points at infinity $E_0=(1\colon 0\colon 0)$ and $E_1=(0\colon 1\colon 0)$ of $\mathcal{C}_1$ are singular. The point $E_0$ is $q^s$-fold with unique tangent line $\ell_Y : Y=0$ having intersection multiplicity $q^s+1$ with $\mathcal{C}_1$ at $E_0$; the point $E_1$ is double with unique tangent line $\ell_X: X=0$ having intersection multiplicity $q^s+1$ with $\mathcal{C}_1$ at $E_1$.

We compute the genus of $\C_1$ by applying Theorem \ref{th:hurwitz} to the covering $\varphi_0\colon\C_1\to\mathbb{P}_y^1$, $(X\colon Y\colon V)\mapsto(Y\colon V)$, where $y$ is the coordinate function of $Y$; $\varphi_0$ has degree $2$.
We describe the places of $\mathbb{P}_y^1$ which ramify in $\varphi_0$; we denote by $P_\mu$, $\mu\in\mathbb{K}$, the zero of $y-\mu$ on $\mathbb{P}_y^1$, and by $P_\infty$ the pole of $y$ on $\mathbb{P}_y^1$.
\begin{itemize}
    \item If $\bar{y}\in\mathbb{K}$ satisfies $\bar{y}\ne0$ and $\beta+\mathrm{Tr}_{q^s/2}(\bar{y})\ne0$, then $F_1(X,\bar{y})$ has two distinct roots $\bar{x}_1,\bar{x}_2\in\mathbb{K}$, and hence $P_{\bar y}$ does not ramify in $\varphi_0$.
    \item If $\bar{y}\in\mathbb{K}$ is one of the $q^s/2$ distinct roots of $\beta+\mathrm{Tr}_{q^s/2}(Y)$, then $(\bar{x},\bar{y})\in\varphi_0^{-1}(P_{\bar{y}})$, with $\bar{x}=\sqrt{\bar{y}^{1-q^s}}$. The tangent line to $\C_1$ at $(\bar{x},\bar{y})$ is $\ell_{\bar y}:Y-\bar{y}=0$, having multiplicity intersection $2$ with $\C_1$ at $(\bar{x},\bar{y})$. Hence,
    \[e((\bar{x},\bar{y})\mid P_{\bar y})= e((\bar{x},\bar{y})\mid P_{\bar y})\cdot v_{P_{\bar y}}(y-\bar{y})=v_{(\bar{x},\bar{y})}(y-\bar{y})=2,\]
    so that $P_{\bar y}$ totally ramifies in $\varphi_0$.
    \item The change of coordinates $(X\colon Y\colon V)\mapsto(X\colon V\colon Y)$ maps $E_1$ to the origin $E_2=(0\colon 0\colon 1)=(0,0)$ and $\mathcal{C}_1$ to the curve $\overline{\mathcal{C}}_1$ with affine equation $\overline{F}_1(X,Y)=0$ where \[\overline{F}_1(X,Y)=X^2+X(\beta Y^{q^s+1}+\sum_{i=0}^{sh-1}Y^{q^s+1-2^i}) + Y^{q^s+1}.\]
The point $E_2$ is double for $\overline{\mathcal{C}}_2$ with tangent line $\ell_X\colon X=0$; while this holds, we apply iteratively the quadratic transformation $(X\colon Y\colon V)\mapsto(XV\colon Y^2\colon YV)$ which maps $\overline{\mathcal{C}}_1$ to the curve with equation $\overline{F}_1(XY,Y)/Y^2=0$. After $k$ times, the curve has equation
\[X^2+X(\beta Y^{q^s+1-k}+\sum_{i=0}^{sh-1}Y^{q^s+1-2^i-k})+Y^{q^s+1-2k}=0.\]
Hence, after $q^s/2$ times, the curve has only one affine point on the line $\ell_Y\colon Y=0$, which is a simple point.
This means that $\mathcal{C}_1$ has exactly one place centered at $E_1$, which we identify with $E_1$.
Since the intersection multiplicity of $\C_1$ at $E_1$ with $\ell_\infty\colon V=0$ and $\ell_Y$ is $2$ and $0$ respectively, we have that $v_{E_1}(y)=-2$; see \cite[Theorem 4.36]{HKT}.
Thus, the pole of $y$ in $\mathbb{P}^1_y$ is totally ramified in $\varphi_0$.

\item The unique place centered at $E_2=(0,0)$ is clearly a zero of $y$. The only place of $\mathbb{P}^1_y$ which can be covered by $E_0$ is the zero of $y$. Therefore, the zero of $y$ in $\mathbb{P}^1_y$ is not ramified in $\varphi_0$, and $E_0$ and $E_2$ are the simple zeros of $y$ on $\mathcal{C}_1$.
\end{itemize}
The $q^s/2+1$ ramification places $P'$, namely $E_1$ and $(\bar{x},\bar{y})$ with $\beta+\mathrm{Tr}_{q^2/2}(\bar{y})=0$, are wildly ramified. For each of them, we choose a local parameter $t^\prime$ at $P^\prime$, a local parameter $t$ at the place $P$ lying under $P^\prime$ in $\varphi_0$, and compute $v_{P^\prime}(d\varphi_0^*(t)/dt^\prime)$, where the pull-back $\varphi_0^*$ of $\varphi_0$ is the identity on $\mathbb{K}(\mathbb{P}_y^1)=\mathbb{K}(y)$.
\begin{itemize}
    \item Let $P^\prime=E_1$, lying over $P=P_\infty$. We choose $t=1/y$. From $v_{E_1}(y)=-2$ we get $v_{E_1}(x)=q^s-1$; hence we can choose $t^\prime=1/(xy^{q^s/2})$. By direct computation, the Laurent series of $t$ at $E_1$ with respect to $t^\prime$ is $t= (t^\prime)^2 +(t^\prime)^3 + w$, with $v_{E_1}(w)\geq4$. Thus, $\frac{dt}{dt^\prime}=(t^\prime)^2+\frac{dw}{dt^\prime}$ has valuation $2$ at $E_1$.
    \item Let $P^\prime=(\bar{x},\bar{y})$, lying over $P=P_{\bar y}$ with $\beta+\mathrm{Tr}_{q^s/2}(\bar y)=0$. We choose $t=y-\bar{y}$ and $t^\prime=x-\bar{x}$.
    By direct computation, the Laurent series of $t$ at $P^\prime$ with respect to $t^\prime$ is $t=\frac{\bar{y}^{q^s}}{\bar{x}+1}(t^\prime)^2+\frac{\bar{y}^{q^s}}{\bar{x}^2+1}(t^\prime)^3+w$, with $v_{P^\prime}(w)\geq4$.
    Thus, $\frac{dt}{dt^\prime}=\frac{\bar{y}^{q^s}}{\bar{x}^2+1}(t^\prime)^2+\frac{dw}{dt^\prime}$ has valuation $2$ at $P^\prime$.
\end{itemize}
Theorem \ref{th:hurwitz} now yields
\[ 2g(\C_1)-2=\deg(\varphi_0)\cdot(2g(\mathbb{P}_y^1)-2)+\left(\frac{q^s}{2}+1\right)\cdot2, \]
whence $g(\C_1)=\frac{q^s}{2}$.
\end{proof}

\begin{proposition}\label{prop:C'2}
The curve $\C_2$ has genus $q^s-1$.
\end{proposition}

\begin{proof}
By Lemma \ref{lemma:C'2irr}, $\C_2$ is absolutely irreducible; hence, the covering $\varphi_1:\C_2\to\C_1$, $(X,Z)\mapsto(X,Y=Z^2+Z+\epsilon)$, is an Artin-Schreier covering of degree $2$.
Every place of $\C_1$ which is not a pole of $y-\epsilon$ is unramified in $\varphi_1$. We consider the unique pole of $y-\epsilon$ on $\C_1$, namely $E_1$.
By direct computation, the Laurent series of $y$ at $E_1$ with respect to the local parameter $t^\prime=1/(xy^{q^s/2})$ is $y-\epsilon=(t^\prime)^{-2}+(t^\prime)^{-1}+w$, with $v_{E_1}(w)\geq0$. Choosing $\omega=(t^\prime)^{-1}$, we have $v_{E_1}((y-\epsilon)-(\omega^2+\omega))=v_{E_1}(w)\geq0$. Thus, by Theorem \ref{th:artinschreier}, $E_1$ is unramified in $\varphi_1$. Altogether, the covering $\varphi_1:\C_2\to\C_1$ is unramified. This implies
\[ 2g(\C_2)-2 = \deg(\varphi_1)\cdot(2g(\C_1)-2), \]
whence $g(\C_2)=q^s-1$.
\end{proof}

\begin{theorem}\label{th:C'} For the curve $\mathcal{C}$ the following holds.
\begin{itemize}
    \item[(a)] $\C$ is absolutely irreducible with genus $q^{2s}-q^s-1$.
    \item[(b)] Let $s$ and $z$ be the coordinate functions of $\mathcal{C}$, and let $t=s^2(z^2+z+\epsilon)$.
The number of $\mathbb{F}_{q^n}$-rational places of $\C$ which are zeros of $t$, or poles of either $s$ or $z$ or $t$, is at most $2q^s+2$.
\end{itemize}
\end{theorem}

\begin{proof}
We compute the valuation of $x$ at the places of $\C_2$. From the proofs of Propositions \ref{prop:C'1} and \ref{prop:C'2} follows that $x$ has exactly $2$ zeros on $\C_1$, namely $E_2$ with $v_{E_2}(x)=1$ and $E_1$ with $v_{E_1}(x)=q^s-1$ ; also, $x$ has exactly $4$ zeros on $\C_2$, namely the two places $Q_1,Q_2$ lying over $E_2$ and the two places $Q_3,Q_4$ over $E_1$. Using the ramification indices, this implies $v_{Q_1}(x)=v_{Q_2}(x)=1$ and $v_{Q_3}(x)=v_{Q_4}(x)=q^s-1$.
The unique pole of $x$ on $\C_1$ is $E_0$; since $v_{E_0}(y)=1$, this implies $v_{E_0}(x)=-q^s$.
The poles of $x$ on $\C_2$ are the two places $R_1,R_2$ lying over $E_0$, with $v_{R_1}(x)=v_{R_2}(x)=-q^s$.

Therefore, by Theorem \ref{th:kummer}, $\C$ is absolutely irreducible and $\varphi_2:\C\to\C_2$ is a Kummer covering of degree $q^s-1$. The places of $\C_2$ which ramify in $\varphi_2$ are exactly $Q_1,Q_2,R_1,R_2$, and they are totally ramified; any other place is unramified in $\varphi_2$.
The genus of $\C$ is
\[g(\C)=1+\deg(\varphi_2)\cdot(g(\C_2)-1)+\frac{1}{2}\cdot 4\cdot(\deg(\varphi_2)-1)=q^{2s}-q^s-1.\]
Using the proofs of Propositions \ref{prop:C'1} and \ref{prop:C'2}, we obtain that:
\begin{itemize}
    \item $s$ has exactly $2$ poles on $\C$, namely the places over $R_1$ or $R_2$;
    \item $z$ has exactly $2(q^s-1)$ poles on $\C$, namely the places over $Q_3$ or $Q_4$;
    \item $s^2$ has exactly $2(q^s-1)+2$ zeros on $\C$; namely, two of them lie over $Q_1$ or $Q_2$, while $2(q^s-1)$ of them lie over $Q_3$ or $Q_4$ (and have been been already considered above);
    \item $z^2+z+\epsilon=y$ has exactly $4$ zeros on $\C$, namely the places over $Q_1,Q_2,R_1,R_2$ (which have been already considered above).
\end{itemize}
Altogether, the number of $\mathbb{F}_{q^n}$-rational places of $\C$ which are poles of $s$, $z$, $t$, or are zeros of $t$, is smaller than or equal to $2q^s+2$.
\end{proof}

From Theorems \ref{th:main} and \ref{th:C'} follows Corollary \ref{cor:mainmain_qeven}, which is our main result Theorem \ref{th:mainmain} when $q$ is even.

\begin{corollary}\label{cor:mainmain_qeven}
Let $q$ be an even prime power, $s\geq1$ be such that $\gcd(s,n)=1$.
Suppose that
\[
n\geq\begin{cases} 4s+2 & \textrm{if}\;q=2\textrm{ and }s>2; \\ 4s+1 & \textrm{otherwise}. \end{cases}
\]
Then for every $\delta\in\mathbb{F}_{q^{2n}}$ satisfying $\mathrm{N}_{q^{2n}/q^n}(\delta)\notin\{0,1\}$ there exists $a\in\mathbb{F}_{q^{2n}}^*$ such that $\dim_{\mathbb{F}_q}\ker(f_{a,b,s})=2$, where $b=\delta a$.
\end{corollary}

\begin{proof}
By Theorems \ref{th:hasseweil} and \ref{th:C'}{\it{(a)}}, the number $N_{q^n}$ of $\mathbb{F}_{q^n}$-rational places of $\mathcal{C}$ satisfies
\[
N_{q^n}\geq q^n+1 - 2(q^{2s}-q^s-1)\sqrt{q^n} > 2q^s+2.
\]
By Theorem \ref{th:C'}{\it{(b)}}, there exists an $\mathbb{F}_{q^n}$-rational affine poin $(\bar{s},\bar{z})$ of $\mathcal{C}$ such that $\bar{t}=\bar{s}^2 (\bar{z}^2+\bar{z}+\epsilon)$ is different from zero. The claim follows.
\end{proof}

\section{Applications to linear sets and rank metric codes}\label{sec:appl}

\subsection{Linear sets}\label{sec:linearsets}

Let $\Lambda=\PG(V,\F_{q^m})=\PG(1,q^m)$, where $V$ is a vector space of dimension $2$ over $\F_{q^m}$.
A point set $L$ of $\Lambda$ is said to be an \emph{$\F_q$-linear set} of $\Lambda$ of rank $k$ if it is
defined by the non-zero vectors of a $k$-dimensional $\F_q$-vector subspace $U$ of $W$, i.e.
\[L=L_U=\{\la {\bf u} \ra_{\mathbb{F}_{q^m}} \colon {\bf u}\in U\setminus \{{\bf 0} \}\}.\]
We say that two linear sets $L_U$ and $L_W$ of $\Lambda=\PG(1,q^m)$ are $\mathrm{P}\Gamma \mathrm{L}$-\emph{equivalent} if there exists $\varphi \in \mathrm{P}\Gamma \mathrm{L} (2,q^m)$ such that $\varphi(L_U)=L_W$.

\smallskip

We start by pointing out that if the point $\langle (0,1) \rangle_{\F_{q^m}}$ is not contained in a linear set $L_U$ of rank $m$ of $\PG(1,q^m)$ (which we can always assume after a suitable projectivity), then $U=U_f=\{(x,f(x))\colon x\in \F_{q^m}\}$ for some $q$-polynomial $\displaystyle f(x)=\sum_{i=0}^{m-1}a_ix^{q^i}\in \tilde{\mathcal{L}}_{m,q}$. In this case we will denote the associated linear set by $L_f$.
Also, recall that the \emph{weight of a point} $P=\langle \mathbf{u} \rangle_{\F_{q^m}}$ is $w_{L_U}(P)=\dim_{\F_q}(U\cap\langle \mathbf{u} \rangle_{\F_{q^m}})$.

\smallskip

One of the most studied classes of linear sets of the projective line, especially because of its applications (see e.g. \cite{Polverino,Sheekey2016}), is the family of maximum scattered linear sets. A {\it maximum scattered} $\F_q$-linear set of $\PG(1,q^m)$ is an $\F_q$-linear set of rank $m$ of $\PG(1,q^m)$ of size $(q^m-1)/(q-1)$, or equivalently a linear set of rank $m$ in $\PG(1,q^m)$ all of whose points have weight one.
If $L_f$ is a maximum scattered linear set in $\PG(1,q^m)$, we also say that $f$ is a \emph{scattered polynomial}.
The known scattered polynomials of $\F_{q^m}$ are
\begin{enumerate}
  \item $f_1(x)=x^{q^s}\in \tilde{\mathcal{L}}_{m,q}$, with $\gcd(s,m)=1$, see \cite{BL2000};
  \item $f_2(x)= x^{q^s}+\alpha x^{q^{m-s}}\in\tilde{\mathcal{L}}_{m,q}$, with $m\geq 4$, $\gcd(s,m)=1$, $\N_{q^m/q}(\alpha) \notin\{0,1\}$, see \cite{LMPT2015,LP2001,Sheekey2016};
  \item $f_3(x)= x^{q^s}+\alpha x^{q^{s+\frac{m}2}}\in\tilde{\mathcal{L}}_{m,q}$, $m \in \{6,8\}$, $\gcd(s,\frac{m}2)=1$ and some conditions on $\alpha$, see \cite{CMPZ} and below;
  \item $f_4(x)=x^q+x^{q^3}+\alpha x^{q^5}\in \tilde{\mathcal{L}}_{6,q}$, $q$ odd and $\alpha^2+\alpha=1$, see \cite{CsMZ2018,MMZ};
  \item $f_5(x)=h^{q-1}x^q-h^{q^2-1}x^{q^2}+x^{q^4}+x^{q^5}\in \tilde{\mathcal{L}}_{6,q}$, $q$ odd, $h^{q^3+1}=-1$, see \cite{BZZ,ZZ}.
\end{enumerate}

\smallskip

In \cite{CMPZ}, the authors introduced the family of linear sets $L_{\delta,s}$ of rank $2n$ in $\PG(1,q^{2n})$ mentioned in 3., i.e. those linear sets defined by the $\F_q$-subspace
\begin{equation}\label{eq:Ud,s}
U_{\delta,s}=\{ (x,f_{\delta,s}(x)) \colon x \in \F_{q^{2n}}  \}\subset \F_{q^{2n}}\times \F_{q^{2n}},
\end{equation}
where
\[
f_{\delta,s}(x)=x^{q^s}+\delta x^{q^{n+s}}\in{\tilde \cL}_{2n,q},
\]
with $\mathrm{N}_{q^{2n/q^n}}(\delta) \notin\{0,1\}$, $1 \leq s \leq 2n-1$ and $\gcd(s,n)=1$.
The relevance of this family relies on the property that each point of $L_{\delta,s}$ has weight at most two; see \cite[Proposition 4.1]{CMPZ}. In \cite[Section 7]{CMPZ} the authors proved that for $n=3$ and $q>4$ there exists $\delta \in \F_{q^2}$ such that $L_{s,\delta}$ is scattered; for $n=4$, $q$ odd and $\delta^2=-1$ the linear set $L_{\delta,s}$ is scattered.
In \cite[Theorem 7.3]{PZ2019} the authors completely determined for $n=3$ necessary and sufficient conditions on $\delta$ ensuring $L_{\delta,s}$ to be scattered.
Note that for $n=3$ we may restrict to the case $s=2$, since every linear set $L_{\delta,s}$ is equivalent to $L_{\delta',2}$ for some $\delta'\in \F_{q^{2n}}^*$.
More precisely, if $\N_{q^6/q^3}(\delta)\notin\{0,1\}$ and we denote $A=-\frac{1}{\delta^{q^3+1}-1}$, one has that $L_{\delta,2}$ is scattered if and only if the equation
\begin{equation}\label{eq:eq2degree} Y^2-(\mathrm{Tr}_{q^3/q}(A)-1)Y+\N_{q^3/q}(A)=0
\end{equation}
admits two distinct roots in $\F_q$.

\begin{theorem}\label{th:noscatt}
Let $q$ be a prime power and $n,s$ be two relatively prime positive integers. Suppose that
\[
n\geq\begin{cases} 4s+2 & \textrm{if}\; q=3\textrm{ and }s>1,\,\textrm{or}\;q=2\textrm{ and }s>2; \\ 4s+1 & \textrm{otherwise}. \end{cases}
\]
Then, for every $\delta\in\mathbb{F}_{q^{2n}}^*$, the $\mathbb{F}_q$-linear set $L_{\delta,s}$ in $\PG(1,q^{2n})$ is not scattered.
\end{theorem}

\begin{proof}
For every $m\in\mathbb{F}_{q^{2n}}$, the weight of the point $\langle(1,m)\rangle_{\mathbb{F}_{q^{2n}}}$ in $L_{\delta,s}$ coincides with the dimension over $\mathbb{F}_q$ of the kernel of $f_{\delta,s}(x)-mx$.

If $\N_{q^{2n}/q^n}(\delta)=1$, then the point $\langle(1,0)\rangle_{\mathbb{F}_{q^{2n}}}$ has weight $n$ in $L_{\delta,s}$.
Let $\N_{q^{2n}/q^n}(\delta)\ne1$. By Theorem \ref{th:mainmain}, there exists $a\in\F_{q^{2n}}^*$ such that $\dim_{\F_q}\ker(f_{a,\delta a,s}(x))=2$, whence
\[
\dim_{\F_q}\ker\left(a\left(f_{\delta,s}(x)+\frac{1}{a}x\right)\right)=2.
\]
This implies that the point $\langle\left(1,-\frac{1}{a}\right)\rangle_{\F_{q^{2n}}}$ has weight $2$ in $L_{\delta,s}$.
The claim is proved.
\end{proof}

Hence, we have the following description for the linear set $L_{\delta,s}$.

\begin{corollary}\label{cor:classbin}
Let $q$ be a prime power and $n,s$ be two relatively prime positive integers.
\begin{itemize}
    \item If $n=3$, then $L_{\delta,s}$ is a scattered linear set if and only if Equation \ref{eq:eq2degree} admits two distinct roots in $\F_q$.
    \item If $n=4$, $q$ is odd and $\delta^2=-1$ then $L_{\delta,s}$ is scattered.
    \item If \[
n\geq\begin{cases} 4s+2 & \textrm{if}\; q=3\textrm{ and }s>1,\,\textrm{or}\;q=2\textrm{ and }s>2, \\ 4s+1 & \textrm{otherwise}, \end{cases}
\]
then, for every $\delta\in\mathbb{F}_{q^{2n}}^*$, $L_{\delta,s}$ is not scattered.
\end{itemize}
\end{corollary}

\begin{proof}
The claim follows from \cite[Theorem 7.3]{PZ2019}, \cite[Theorem 7.2]{CMPZ}, and Theorem \ref{th:noscatt}.
\end{proof}

Among the known scattered polynomials listed above, the families in 3., 4. and 5. provide scattered polynomials for infinitely many $q$'s, but only over a specific extension of $\F_q$, namely either $\F_{q^6}$ or $\F_{q^8}$.
Unlike this situation, the families in 1.\ and 2.\ provide scattered polynomials over infinitely many extensions $\F_{q^m}$ of $\F_{q}$; they are named respectively as scattered polynomials of pseudoregulus type, and as scattered polynomials of LP type (after Lunardon and Polverino).

The scattered polynomials of pseudoregulus or LP type have raised the following question: which polynomials over $\mathbb{F}_{q^m}$ are scattered over infinitely many extensions of $\mathbb{F}_{q^m}$?

\begin{definition}{\rm \cite[Section 1]{BZ}}
Let $f(x)\in\tilde{\cL}_{m,q}$, $0\leq t\leq m-1$, $\ell\geq1$, and $U_{\ell}=\{(x^{q^t},f(x))\colon x\in\F_{q^{m\ell}}\}$.
We say that $f(x)$ is an exceptional scattered polynomial of index $t$ if $L_{U_{\ell}}$ is a scattered $\F_q$-linear set in $\PG(1,q^{m\ell})$ for infinitely many $\ell$'s.
\end{definition}

Clearly, the scattered polynomials of pseudoregulus type are exceptional scattered of index $0$.
Also, for the scattered polynomial $f_2(x)$ of LP type,
\[
U_{f_2}=\{(x^{q^s}, x^{q^{2s}}+\alpha x)\colon x\in\F_{q^m}\};
\]
thus, the polynomial $x^{q^{2s}}+\alpha x$ is exceptional scattered of index $s$.

For a scattered polynomial $f(x)\in\tilde{\cL}_{m,q}$ of index $t$, we say that $f(x)$ is $t$-normalized if the following properties hold: $f(x)$ is monic; the coefficient of $x^{q^t}$ in $f(x)$ is zero; if $t>0$, the coefficient of $x$ in $f(x)$ is nonzero.
Up to ${\rm PGL}$-equivalence of the corresponding scattered linear set, we may always assume that $f(x)$ is $t$-normalized.
\begin{theorem}
Let $f(x)\in\tilde{\cL}_{m,q}$ be a $t$-normalized exceptional scattered polynomial of index $t$. Then the following holds.
\begin{itemize}
    \item If $t=0$, then $f(x)$ is of pseudoregulus type; see \cite[Corollary 3.4]{BZ} for $q>5$, \cite[Section 4]{BM} for $q\leq5$.
    \item If $t=1$ or $t=2$, then $f(x)$ is either of pseudoregulus type or of LP type; see \cite[Corollary 3.7]{BZ} for $t=1$, \cite[Corollary 1.4]{BM} for $t=2$.
    \item If $t\geq3$, $q$ is odd, and $\max\{\deg_q f(x),t\}$ is an odd prime, then $f(x)=x$; see \cite[Theorem 1.2]{FM}.
\end{itemize}
\end{theorem}

Recall that the polynomials $f_3(x)$ of family 3. in the list above are scattered under certain assumptions for $m\in\{6,8\}$; even when $f_3(x)$ is not scattered, still all the points of $L_{f_3}$ have weight at most $2$.
Thus, one may conjecture that family 3. contains scattered polynomials over $\F_{q^m}$ for every even $m$. Note that, even if this is the case, the arising scattered polynomials are not exceptional: not only the coefficients but also the degree depend heavily on the underlying field $\F_{q^m}$.

Our asymptotic result Theorem \ref{th:mainmain} shows that the family of scattered polynomial in 3. cannot be extended to any higher extension $\mathbb{F}_{q^m}$ when $m$ is large enough with respect to $s$.


\subsection{Rank metric codes}\label{sec:MRD}

Rank metric codes were introduced by Delsarte \cite{Delsarte} in 1978 and they have been intensively investigated in recent years because of their applications; we refer to \cite{sheekey_newest_preprint} for a recent survey on this topic.
The set of $m \times n$ matrices $\fq^{m\times n}$ over $\fq$ may be endowed with a metric, called \emph{rank metric}, defined by
\[d(A,B) = \mathrm{rk}\,(A-B).\]
A subset $\C \subseteq \fq^{m\times n}$ equipped with the rank metric is called a \emph{rank metric code} (shortly, a \emph{RM}-code).
The minimum distance of $\C$ is defined as
\[d = \min\{ d(A,B) \colon A,B \in \C,\,\, A\neq B \}.\]
Denote the parameters of a RM-code $\C\subseteq\fq^{m,n}$ with minimum distance $d$ by $(m,n,q;d)$.
We are interested in $\fq$-\emph{linear} RM-codes, i.e. $\fq$-subspaces of $\fq^{m\times n}$.
Delsarte showed in \cite{Delsarte} that the parameters of these codes must obey a Singleton-like bound, i.e.
\[ |\C| \leq q^{\max\{m,n\}(\min\{m,n\}-d+1)}. \]
When equality holds, we call $\C$ a \emph{maximum rank distance} (\emph{MRD} for short) code.
Examples of $\fq$-linear MRD-codes were first found in \cite{Delsarte,Gabidulin}.
We say that two $\fq$-linear RM-codes $\C$ and $\C'$ are equivalent if there exist $X \in \mathrm{GL}(m,q)$, $Y \in \mathrm{GL}(n,q)$, and $\sigma\in{\rm Aut}(\fq)$ such that
\[\C'=\{XC^\sigma Y \colon C \in \C\}.\]

The \emph{left} and \emph{right} idealisers of $\C$ are defined in \cite{LN2016} as $L(\C)=\{A \in \mathrm{GL}(m,q) \colon A \C\subseteq \C\}$ and $R(\C)=\{B \in \mathrm{GL}(n,q) \colon \C B \subseteq \C\}$. They are invariant under the equivalence of rank metric codes, and have been investigated in \cite{LTZ2}; further invariants have been introduced in \cite{GZ,NPH2}.

Much of the focus on MRD-codes of $\fq^{m\times m}$ to date has been on codes which are $\F_{q^m}$-\emph{linear}, i.e. codes in which the
left (or right) idealiser contains a field isomorphic to $\F_{q^m}$, since for such codes a fast decoding algorithm has been developed in \cite{Gabidulin}.
Very few examples of such codes are known, see \cite{BZZ,CMPZ,CsMPZh,CsMZ2018,Delsarte,Gabidulin,LP2001,MMZ,Sheekey2016,ZZ}.

In \cite[Section 5]{Sheekey2016} Sheekey showed that scattered $\F_q$-linear sets of $\PG(1,q^m)$ of rank $m$ yield $\F_q$-linear MRD-codes with parameters $(m,m,q;m-1)$ with left idealiser isomorphic to $\F_{q^m}$; see \cite{CsMPZ2019,CSMPZ2016,ShVdV} for further details on such kind of connections.
We briefly recall here the construction from \cite{Sheekey2016}. Let $U_f=\{(x,f(x))\colon x\in \F_{q^m}\}$, where $f(x)$ is a scattered $q$-polynomial.
The choice of an $\F_q$-basis for $\F_{q^m}$ defines a canonical ring isomorphism between $\mathrm{End}(\F_{q^m},\F_q)$ and $\F_q^{m\times m}$.
Thus, the set
\[
\C_f=\{x\mapsto af(x)+bx \colon a,b \in \F_{q^m}\}\subset \mathrm{End}(\F_{q^m},\F_q)
\]
corresponds to a set of $m\times m$ matrices over $\F_q$ forming an $\F_q$-linear MRD-code with parameters $(m,m,q;m-1)$. Also, as $\C_f$ is an $\F_{q^m}$-subspace of $\mathrm{End}(\F_{q^m},\F_q)$,
its left idealiser $L(\C_f)$ is isomorphic to $\F_{q^m}$; see also \cite[Section 6]{CMPZ}.

Now consider the set
\[ \C_{f_{\delta,s}}=\{ x\mapsto a(x^{q^s}+\delta x^{q^{s+n}})+bx \colon a,b \in \F_{q^{2n}} \}, \]
which corresponds to a set of $2n\times 2n$ matrices over $\F_q$ forming an $\F_q$-linear rank metric code with parameters $(2n,2n,q;2n-i)$, where
\[ i=\max\{ w_{L_{\delta,s}}(P) \colon P \in \PG(1,q^{2n}) \}. \]

The following theorem is a consequence of Corollary \ref{cor:classbin} and states that, when $n$ is large enough, $\C_{f_{\delta,s}}$ is not an MRD-code.

\begin{theorem}\label{th:applMRD}
Let $q$ be a prime power and $n,s$ be two relatively prime positive integers.
\begin{itemize}
    \item If $n=3$, then $\C_{f_{\delta,s}}$ is an MRD-code if and only if Equation {\rm \ref{eq:eq2degree}} admits two distinct roots in $\F_q$; see {\rm \cite{CMPZ}} and {\rm \cite{PZ2019}}.
    \item If $n=4$, $q$ odd and $\delta^2=-1$ then $\C_{f_{\delta,s}}$ is an MRD-code; see {\rm \cite{CMPZ}}.
    \item If \[
n\geq\begin{cases} 4s+2 & \textrm{if}\; q=3\textrm{ and }s>1,\,\textrm{or}\;q=2\textrm{ and }s>2 \\ 4s+1 & \textrm{otherwise}; \end{cases}
\]
then, for every $\delta\in\mathbb{F}_{q^{2n}}^*$, $\C_{f_{\delta,s}}$ is not an MRD-code and its minimum distance is $n-2$.
\end{itemize}
\end{theorem}

\bigskip
\par\noindent Olga Polverino, Giovanni Zini and Ferdinando Zullo\\
Dipartimento di Matematica e Fisica,\\
Universit\`a degli Studi della Campania ``Luigi Vanvitelli'',\\
Viale Lincoln 5,\\
I--\,81100 Caserta, Italy\\
{{\em \{olga.polverino,giovanni.zini,ferdinando.zullo\}@unicampania.it}}

\end{document}